\documentclass[a4paper]{amsart}
\usepackage{amssymb}
\usepackage{url}
\usepackage[all]{xy}
\usepackage{stmaryrd}

\theoremstyle{plain}
\newtheorem{prop}{Proposition}
\newtheorem{thm}[prop]{Theorem}
\newtheorem{cor}[prop]{Corollary}

\newtheorem{fact}[prop]{Fact}

\newtheorem*{thmA}{Theorem A}

\theoremstyle{definition}

\theoremstyle{remark}

\newtheorem{rem}[prop]{Remark}

\numberwithin{prop}{section} 
\numberwithin{equation}{section}
\numberwithin{claim}{prop}

\DeclareMathOperator{\spn}{span}
\DeclareMathOperator{\alg}{alg}
\DeclareMathOperator{\mmod}{mod}

\DeclareMathOperator{\Hom}{Hom}
\DeclareMathOperator{\End}{End}
\DeclareMathOperator{\Ext}{Ext}
\DeclareMathOperator{\image}{im}
\DeclareMathOperator{\kernel}{ker}

\DeclareMathOperator{\iid}{id}
\DeclareMathOperator{\sgn}{sgn}
\DeclareMathOperator{\Mind}{\mathbf{ind}}

\DeclareMathOperator{\trace}{tr}
\DeclareMathOperator{\eval}{ev}
\DeclareMathOperator{\dHom}{\underline{\Hom}}
\DeclareMathOperator{\dExt}{\underline{\Ext}}
\DeclareMathOperator{\utimes}{\underline{\otimes}}
\DeclareMathOperator{\Cone}{Cone}

\newcommand{\der}{\partial}
\newcommand{\eps}{\varepsilon}

\newcommand{\ca}[1]{\mathcal{#1}}

\newcommand{\N}{\mathbb{N}}
\newcommand{\Z}{\mathbb{Z}}

\newcommand{\euV}{\mathfrak{V}}
\newcommand{\euE}{\mathfrak{E}}

\newcommand{\boA}{\mathbf{A}}
\newcommand{\boB}{\mathbf{B}}
\newcommand{\uboA}{\underline{\boA}}
\newcommand{\uboB}{\underline{\boB}}

\newcommand{\dbr}{\rrbracket}
\newcommand{\dbl}{\llbracket}

\newcommand{\caK}{\ca{K}}

\newcommand{\caH}{\ca{H}}

\newcommand{\Hmod}{{\caH\mbox{-}\!\mmod}}
\newcommand{\HImod}{{\caH_I\mbox{-}\!\mmod}}

\newcommand{\Ralg}{{R\mbox{-}\!\alg}}

\newcommand{\caB}{\mathcal{B}}

\newcommand{\caC}{\ca{C}}
\newcommand{\caP}{\ca{P}}

\newcommand{\argu}{\underline{\phantom{x}}} 
\newcommand{\FP}{{\rm FP}}

\begin{document}
\title{The Euler characteristic of a Hecke algebra}
\author{T.~Terragni \and Th.~Weigel}
\date{\today}
\address{T.~Terragni, Th.~Weigel\\
  Universit\`a di Milano-Bicocca\\
  Ed.~U5, Via R.Cozzi, 55\\
  20125 Milano, Italy}
\email{tom.terragni@gmail.com}
\email{thomas.weigel@unimib.it}


\begin{abstract}
It is shown that the Euler characteristic $\chi_{(\caH,\caB,\eps_q)}$ of a $\Z\dbl q\dbr$-Hecke algebra $\caH$ associated with a finitely generated Coxeter group $(W,S)$ 
coincides with $p_{(W,S)}(q)^{-1}$, where $p_{(W,S)}(t)$ is the Poincar\'e series of $(W,S)$.
\end{abstract}
\maketitle
%%%%%%%%%%%%%%%%%%%%%%%%%%%%%%%%%%% 

\section{Introduction}
\label{s:intro}
For a finitely generated {\it Coxeter group} $(W,S)$ one defines the {\it Poincar\'e series} by
\begin{equation}
\label{eq:poinser}
p_{(W,S)}(t)=\sum_{w\in W} t^{\ell(w)}\in\Z\dbl t\dbr,
\end{equation}
where $\ell\colon W\to\N_0$ denotes the {\it length function} associated with $(W,S)$.
It is well known that $p_{(W,S)}$ is a rational function in $t$ (cf.~\cite[Chap.~IV, \S1, Ex.~25 and 26]{bourbaki--gal46}). 
This function is explicitly known for finite Coxeter groups, and explicitly computable for any given infinite, finitely generated Coxeter group $(W,S)$ using the recursive sum formula 
\begin{equation}
\label{eq:altsfor}
\frac{1}{p_{(W,S)}(t)}=\sum_{I\subsetneq S} (-1)^{|S|-|I|-1}\frac{1}{p_{(W_I,I)}(t)}
\end{equation}
(cf. \cite[\S 5.12]{humphreys--rgcg}). The formula \eqref{eq:altsfor} suggests that one should be able to interpret ${p_{(W,S)}(t)}^{-1}$ as an Euler characteristic
in a suitable context. 
This point of view is also supported by a result of J-P.~Serre who showed that $p_{(W,S)}(1)^{-1}$ coincides with the Euler characteristic of the Coxeter group $W$ (cf. \cite[\S 1.9]{serre--cgd}).

The main goal of this paper is to introduce the notion of an Euler characteristic for certain augmented algebras with a distinguished basis, and to show that ${p_{(W,S)}(q)}^{-1}$ coincides with the Euler characteristic of the $\Z\dbl q\dbr$-\emph{Hecke algebra} $\caH_q(W,S)$ endowed with the standard basis $\caB=\{\,T_w\mid w\in W\,\}$ and the index representation $\eps_q\colon \caH_q(W,S)\to \Z\dbl q\dbr$.

Let $R$ be a commutative ring, $\boA$ an $R$-algebra, $\caB$ a subset of $\boA$ containing $1_\boA$ such that $\boA$ is a free $R$-module spanned by $\caB$, and let $\lambda\colon \boA\to R$ be an $R$-algebra homomorphism satisfying
\begin{itemize}
\item[(i)] the left $\boA$-module $R_\lambda$ associated with $\lambda$ is of type \FP, and
\item[(ii)] the $R$-linear map $\tilde\mu\colon \boA\to R$ given by $\tilde{\mu}(1)=1$, $\tilde\mu(b)=0$ for all $b\in \caB\setminus\{1\}$ satisfies $\tilde\mu(xy)=\tilde\mu(yx)$ for all $x,y\in \boA$.
\end{itemize}
Then one may define the \emph{Euler characteristic} $\chi_{(\boA, \caB, \lambda)}\in R$ as the value of the trace function associated with $\tilde\mu$ on the Hattori--Stallings rank (cf.~\cite[\S{}IX.2]{brown--cg}) of the left $\boA$-module $R_\lambda$ (cf.~\S\ref{ss:eueuc}).
For short we call a triple $(\boA, \caB,\lambda)$ satisfying the conditions (i) and (ii) an \emph{Euler algebra}, e.g., $(\Z[G],G, \eps)$ for a group $G$ of type \FP\ is an Euler algebra, and $\chi_{(\Z[G],G, \eps)}$ coincides with the Euler characteristic of $G$ (cf.~\cite[\S{}IX.6]{brown--cg}).
Our main result (cf.~Thm.~\ref{thm:AA}) can be stated as follows.
\begin{thmA}
Let $(W,S)$ be a finitely generated Coxeter group, and let $\caH=\caH_q(W,S)$ be the $\Z\dbl q\dbr$-Hecke algebra associated with $(W,S)$.
Then $(\caH,\caB,\eps_q)$ is a $\Z\dbl q\dbr$-Euler algebra, and $\chi_{(\caH,\caB,\eps_q)}=p_{(W,S)}(q)^{-1}$.
\end{thmA}

In order to show the property (i) and to compute the Hattori--Stallings rank of the augmented algebra $(\caH,\eps_q)$ we will make use of a chain complex $C=(C_\bullet,\der_\bullet)$ of left $\caH$-modules first established by V.V.~Deodhar in \cite{deodhar--sgabo2}.

\begin{rem} The Poincar\'e series of $(W,S)$ (cf.~\eqref{eq:poinser}) coincides with the Poincar\'e series of ${(\caH,\caB,\eps_q)}$, given by
\begin{equation}
\label{eq:poinalt}
p_{(\caH,\caB,\eps_q)}=\sum_{b\in\caB} \eps_q(b)\in\Z\dbl q\dbr,
\end{equation}
i.e., it involves only combinatorial data of $(\caH,\caB,\eps_q)$.
In particular, one has the identity $p_{(\caH,\caB,\eps_q)}\cdot\chi_{(\caH,\caB,\eps_q)}=1$ in $\Z\dbl q\dbr$.
A similar identity involving combinatorial data and cohomological data is known for Koszul algebras (cf. \cite[p.~22, Cor.~2.2]{polishchuk-positselski--qa}).
It would be interesting to know whether there exist other examples of $\Z\dbl q\dbr$-Euler algebras $(\boA,\caB,\lambda)$ for which $p_{(\boA,\caB,\lambda)}$ given by \eqref{eq:poinalt} is defined and which satisfy the identity $p_{(\boA,\caB,\lambda)}\cdot\chi_{(\boA,\caB,\lambda)}=1$.
\end{rem}
\subsection*{Acknowledgments}
The authors would like to thank F.~Brenti for some very helpful discussions, and M.~Solleveld for pointing out that in a slightly different context projective resolutions of affine Hecke algebras were already constructed in~\cite{opdam-solleveld--haaha}.
Our gratitude goes also to A.~Mathas and S.~Schroll for informing us about the Deodhar complex (cf.~\cite{deodhar--sgabo2}, \cite{mathas--qacc}) and its relation to the complex of $(\caH,\caH)$-bimodules established in \cite{linckelmann-schroll--tsqacc}.

\noindent The present paper partly comes from the first author's PhD thesis \cite{terragni--haacg}.

\section{Coxeter groups and Hecke algebras}
\label{s:coxhec}

\subsection{Coxeter groups}
\label{ss:cox}
A \emph{Coxeter graph} $\Gamma=(\euV,\euE,m)$ is a finite combinatorial graph $(\euV,\euE)$ with vertex set%
\footnote{In this context the graph $\emptyset$ with empty vertex set is also considered as a Coxeter graph.} 
$\euV$ and non-oriented edges (i.e., two-element subsets of $\euV$) $\{i,j\}\in\euE\subseteq\caP_2(\euV)$ labelled by positive integers $m_{i,j}\geq 3$ or infinity.

The \emph{Coxeter group} $(W,S)$ associated with $\Gamma=(\euV,\euE,m)$ consists of the group $W$ generated by the set of involutions $S=\{\,s_i\mid i\in\euV\,\}$ subject to the relations $(s_is_j)^{m_{i,j}}=1$, where $\{i,j\}\in\euE$ is an edge of label $m_{i,j}<\infty$, and the commutation relations $s_is_j=s_js_i$ whenever $\{i,j\}\not\in \euE$.
The \emph{length function} on $W$ with respect to $S$ will be denoted by $\ell\colon W\to\N_0$. Since $S=S^{-1}$ is a set of involutions, $\ell(w)=\ell(w^{-1})$, and it is well known that a longest element $w_0\in W$ exists if, and only if, $W$ is finite. 
In this case it is unique and has the property that $\ell(w_0 v)=\ell(w_0)-\ell(v)$ for all $v\in W$.
A Coxeter group which is finite is called \emph{spherical}, and \emph{non-spherical} otherwise.
For a subset $I \subseteq S$ let $W_I$ be the corresponding  parabolic subgroup, i.e., $W_I$ is the subgroup of $W$ generated by $I$. 
It is isomorphic to the Coxeter group associated to the Coxeter subgraph $\Gamma^\prime$ of $\Gamma$ based on the vertices $\{\,i\in\euV\mid s_i\in I\,\}$.
The length function of $W$ restricted to $W_I$ coincides with the intrinsic length function of the Coxeter group $(W_I,I)$. 
Put
\begin{equation}
  \label{eq:lcosetreps}
  W^I=\{\, w\in W\mid \ell(ws)>\ell(w) \text{ for all }s\in I\,\},
\end{equation}
and let ${}^IW=(W^I)^{-1}$, i.e.,
\begin{equation}
  \label{eq:rcosetreps}
        {}^IW=\{\, w\in W\mid \ell(sw)>\ell(w) \text{ for all }s\in I\,\}.
\end{equation}
For the reader's convenience we recall the following properties (cf.~\cite[\S5.12]{humphreys--rgcg}).

\begin{prop}
  \label{prop:cosreps}
  Let $(W,S)$ be a Coxeter group, let $w\in W$ and let $I\subseteq S$.
  \begin{itemize} 
  \item[(a)] $W^I$ and ${}^IW$ are sets of coset representatives distinguished in the sense that the decompositions $W=W^IW_I={}_IW{}^IW$ are length-additive:     there exists a unique $4$-tuple $(w_I,w^I,{}_Iw,{}^Iw)\in W_I\times W^I\times W_I\times {}^IW$ such that $w=w^Iw_I={}_Iw{}^Iw$ and $\ell(w) = \ell(w^I) + \ell(w_I) = \ell({}_Iw) + \ell({}^Iw)$.
  \item[(b)] The element $w^I\in W^I$ is the unique shortest element in $wW_I$. 
  \item[(c)] Let $y\in W^I$ and $u\in W_I$. Then $(yu)^I= y$, $(yu)_I=u$, and $\ell(yu)=\ell(y)+\ell(u)$.
  \item[(d)] For $s\in S$ one has $W={}^{\{s\}}W \sqcup s({}^{\{s\}}W)$, where $\sqcup$ denotes disjoint union.
  \item[(e)] Let $I\subseteq J\subseteq S$. Then $W^J\subseteq W^I$. Moreover, $W^S=\{1\}$ and $W^\emptyset=W$. 
  \end{itemize}
\end{prop}

\subsection{Hecke algebras}
\label{ss:hec}
Let $R$ be a commutative ring with unit and with a distinguished%
\footnote{For certain Coxeter groups it is also possible to consider multiple parameter Hecke algebras (cf.~\cite{terragni--haacg}).} 
element $q\in R$.
The \emph{$R$-Hecke algebra} $\caH=\caH_q(W,S)$ associated with $(W,S)$ and $q$ is the unique associative $R$-algebra which is a free $R$-module with basis $\{\,T_{w}\mid {w}\in W\,\}$ subject to the relations
\begin{equation}
  \label{eq:hecrel}
  T_sT_{w}=
  \begin{cases}
    T_{s{w}}&\qquad\text{if $\ell(s{w})>\ell({w})$},\\
    (q-1)T_{w}+qT_{s{w}}&\qquad\text{if $\ell(s{w})<\ell({w})$,}
  \end{cases}
\end{equation}
for $s\in S$, ${w}\in W$. 
In particular, one has a canonical isomorphism $\caH_1(W,S)\simeq R[W]$, where $R[W]$ denotes the $R$-group algebra of $W$. 
The $R$-algebra $\caH_q(W,S)$ comes equipped with a standard basis $\caB=\{\, T_w\mid w\in W\,\}$ defining the $R$-linear map $\tilde\mu_\caB\colon \caH\to R$ given by $\tilde{\mu}_\caB(T_1)= 1$ and $\tilde\mu_\caB(T_w)=0$ for $w\in W\setminus\{1\}$. 

For $I\subseteq S$ we denote by $\caH_I$ the corresponding parabolic subalgebra, i.e., the $R$-subalgebra of $\caH$ generated by $\{\,T_s\mid s\in I\,\}$ which coincides with the $R$-module spanned by $\caB_I=\{\, T_w\mid w\in W_I\,\}=\caB\cap\caH_I$, e.g., for $I=\emptyset$, one has $\caH_\emptyset=R$.
For further details see \cite[Chap.~7]{humphreys--rgcg}.

\subsection{$\caH$-modules}
\label{ss:one} 
Any $R$-algebra homomorphism $\lambda\in\Hom_{\Ralg}(\caH,R)$ defines a $1$-di\-men\-sion\-al left $\caH$-module $R_\lambda$, i.e., for $T_{w}\in\caH$, ${w}\in W$, and $r\in R_\lambda$ one has $T_{w}.r=\lambda(T_{w})r$. 
By \eqref{eq:hecrel}, one has $\lambda(T_s)\in \{-1,q\}$ for all $s\in S$, and $\lambda(T_{s_i})=\lambda(T_{s_j})$ for $s_i, s_j\in S$ and $m_{i,j}$ odd. 
There are two distinguished $R$-algebra homomorphisms $\eps_q, \eps_{-1}\in \Hom_{\Ralg}(\caH,R)$, respectively, the \emph{index} and \emph{sign} representations (cf.~\cite[8.1.3]{geck-pfeiffer--cfcgiha}), given by $\eps_q(T_s)=q$, and $\eps_{-1}(T_s)=-1$, $s\in S$. 
In the present context, one may consider $\eps_q$ as the \emph{augmentation} of the algebra $\caH$.
Note that $\eps_q(T_w)=q^{\ell(w)}$ and $\eps_{-1}(T_w)=(-1)^{\ell(w)}$. 
For short we put $R_q=R_{\eps_q}$, $R_{-1}=R_{\eps_{-1}}$, and use also the same notation for the restriction of these modules to any parabolic subalgebra.

For $I\subseteq S$ let $\caH^I=\spn_R\{\,T_w\mid w\in W^I\,\}\subseteq \caH$.
Multiplication in $\caH$ induces a canonical map of right $\caH_I$-modules $\caH^I\otimes_R\caH_I\longrightarrow\caH$.
Let $y\in W^I$ and $u\in W_I$. As $\ell(yu)=\ell(y)+\ell(u)$ (cf. Prop.~\ref{prop:cosreps}(c)), one has $T_yT_u= T_{yu}$. 
This shows that this map is an isomorphism. 
In particular, $\caH$ is a projective right $\caH_I$-module and
\begin{equation}
  \label{eq:ind}
  \Mind_I^S(\argu) =\Mind_{\caH_I}^{\caH}(\argu) =
  \caH\otimes_{\caH_I}\argu\colon\HImod\longrightarrow\Hmod
\end{equation}
is an exact functor mapping projectives to projectives. 
Moreover, one has the following.
\begin{fact}
  \label{fact:decoH}
  The canonical map $c_I\colon\caH^I\to\Mind_I^S(R_q)$
  given by 
  $c_I(T_w)=T_w\eta_I$, where $\eta_I=T_1\otimes 1\in \Mind_I^S(R_q)$ 
  and $w\in W^I$, is an isomorphism of $R$-modules.
  Moreover, for $w\in W$, one has $T_w\eta_I=\eps_q(T_{w_I})T_{w^I}\eta_I$.
\end{fact}

In case that $I\subseteq S$ generates a finite group, one has the following.
\begin{prop}
  \label{prop:eI}
  Let $I$ be a subset of $S$ such that $W_I$ is finite. 
  Put $\tau_I=\sum_{w\in W_I}T_w$. 
  Then one has the following:
  \begin{itemize}
  \item[(a)] $\tau_I^2=p_{(W_I,I)}(q)\tau_I$.
  \end{itemize}
  Moreover, if $p_{(W_I,I)}(q)\in R^\times$ is invertible in $R$ and $e_I=(p_{(W_I,I)}(q))^{-1}\tau_I$, then
  \begin{itemize}
  \item[(b)] the element $e_I$ is a central idempotent in $\caH_I$, 
  \item[(c)] the left ideal $\caH e_I$ is a finitely generated, projective, left $\caH$-module isomorphic to $\Mind_I^S(R_q)$,
  \item[(d)] $T_w e_I=\eps_q(T_{w_I})T_{w^I}e_I$.
  \end{itemize}
\end{prop}
\begin{proof} For $s\in I$, put $X_s=\sum_{w\in {}^{\{s\}}(W_I)}T_w$. 
  Then $\tau_I=(T_1+T_s)X_s$ (cf. Prop.~\ref{prop:cosreps}(d)) and therefore  
  \begin{equation*}
    T_s\tau_I= T_s (T_1 + T_s)X_s= \left[T_s + q T_1 + (q-1)T_s\right]X_s=q(T_1+T_s)X_s=\eps_q(T_s) \tau_I.
  \end{equation*}
  This shows (a). 
  Part (b) is an immediate consequence of (a), and the first part of (c) follows from the decomposition of the regular module $\caH = \caH e_I \oplus \caH (T_1-e_I)$.
  The canonical map $\pi\colon\caH\to \Mind^S_I(R_q)$, $\pi(T_w)=T_w\eta_I$, is a surjective morphism of $\caH$-modules with $\kernel(\pi)=\caH(T_1-e_I)$.
  This yields the second part of (c).
  Part (d) follows from part (b) and Proposition~\ref{prop:cosreps}(a).
\end{proof}

\section{The Deodhar complex}
\label{ss:HCcomplex}
There is a chain complex of left $\caH$-modules $C=(C_\bullet,\der_\bullet)$ which can be seen as the module-theoretic analogue of the Coxeter complex associated with a Coxeter group $(W,S)$.
This chain complex has been introduced first by V.V.~Deodhar in \cite{deodhar--sgabo2}.
For spherical Coxeter groups it was studied in more detail by A.~Mathas in \cite{mathas--qacc}, while M.~Linckelmann and S.~Schroll introduced in \cite{linckelmann-schroll--tsqacc} a two-sided version of this complex for spherical Coxeter groups.
The definition of this chain complex is quite technical and depends on the choice of a sign function. 
For the convenience of the reader in this section we recall its definition and basic properties.

\subsection{Sign maps}
\label{ss:sign}
Let  $\caP(S)$ denote the set of subsets of a finite set $S$.
A \emph{sign map} for $S$ is a function $\sgn\colon S\times\caP(S)\to\{\pm1\}$ satisfying 
\begin{equation}\label{eq:signprop}
  \sgn(s,I)\sgn(t,I\sqcup\{s\})+\sgn(t,I)\sgn(s,I\sqcup\{t\})=0
\end{equation}
for all $I\subseteq S$ and $s, t\in S\setminus I$, $s\not=t$.
Such functions do exist; e.g., if ``$<$'' is a total order on the finite set $S$, the function $\sgn(s,I)=(-1)^{|\{\,t\in S\setminus I \, \mid\, t<s\,\}|}$ is a sign map.

\subsection{The Deodhar complex}
\label{ss:indu}
Let $I$ and $J$ be subsets of $S$ satisfying $I\subseteq J\subseteq S$.
The canonical injection $\caH_I\to\caH_J$ is a morphism of augmented $R$-algebras. 
Hence it induces a morphism of left $\caH$-modules $d_I^J\colon \Mind^S_I(R_q)\longrightarrow \Mind^S_{J}(R_q)$ given by
\begin{equation}
  \label{eq:bdry}
    d_I^J(T_w\otimes_{\caH_I} r)=T_{w}\otimes_{\caH_J}r.
\end{equation}
For a subset $I\subseteq S$ put $\deg (I)=|S|-|I|-1$. 
Thus $\deg(I)\in \{\,-1,\dots,|S|-1\,\}$. 
For a non-negative integer $k$ let $C_k$ be the left $\caH$-module
\begin{equation}
  \label{eq:defC}
  C_k = \coprod_{\substack{I\subseteq S\\\deg(I)=k}} \Mind_I^S(R_q),
\end{equation}
and let $\der_k\colon C_k\to C_{k-1}$ be the morphism of left $\caH$-modules given by
\begin{equation}
  \label{eq:derdef1}
  \der_k=\sum_{\substack{I,J\subseteq S\\\deg(I)=k,\\\deg(J)=k-1}}\der_{I,J},
\end{equation}
where
\begin{equation}
  \label{eq:derdef2}
  \der_{I,J}=\begin{cases}
  \sgn(s,I) d_I^J & \text{ if }J=I\sqcup\{s\},\\
  0              & \text{ if }J\not\supseteq I,
  \end{cases}
\end{equation}
and $d_I^J$ is given as in \eqref{eq:bdry}. 
Note that $C_k=0$ for $k> |S|-1$. 
The following properties have been established in \cite[Thm.~5.1]{deodhar--sgabo2}.

\begin{thm}[V.V.~Deodhar]
\label{thm:complex}
Let $\caH=\caH_q(W,S)$ be an $R$-Hecke algebra, and let $C=(C_\bullet, \der_\bullet)$ be as described above. 
Then
\begin{itemize}
\item[{\rm (a)}] $\der_k\circ \der_{k+1}=0$, i.e., $C$ is a chain complex.
\item[{\rm (b)}] If $W$ is finite, then 
  \begin{equation*}
  H_k(C)\simeq \begin{cases}
  R_q & \text{ for }k=0\\
  R_{-1} & \text{ for }k=|S|-1\\
  0 & \text{ otherwise.}
  \end{cases}
  \end{equation*}
\item[{\rm (c)}] If $W$ is infinite, then $C$ has homology concentrated in degree zero with $H_0(C)\simeq R_q$. 
\end{itemize}
\end{thm}

From now on $C=(C_\bullet, \der_\bullet)$ will be called the \emph{Deodhar complex} of $\caH$.
\begin{rem}
  \label{rem:complex}
  Let $C=(C_\bullet,\der_\bullet)$ be the Deodhar complex of $\caH$.
  
  \noindent
  (a) In degree $|S|-1$, $C_{|S|-1}=\Mind_\emptyset^S(R_q)\simeq \caH$ coincides with the regular left $\caH$-module.
  
  \noindent
  (b) Let $\eps=\sum_{s\in S}\der_{S\setminus\{s\},S}\colon C_0\to R_q=\Mind_S^S(R_q)$ denote the canonical map given by~\eqref{eq:derdef2}. 
  Then the chain complex of left $\caH$-modules
  \begin{equation}
    \label{eq:trivhom}
    \xymatrix{
      \ldots\ar[r]&C_2\ar[r]^{\der_2}&C_1\ar[r]^{\der_1}&C_0\ar[r]^{\eps}&R_q\ar[r]&0}
  \end{equation}
  is acyclic for $W$ infinite.
\end{rem}

\section{Euler algebras}
\label{s:Euler}

Let $\boA$ be an associative $R$-algebra (with unit $1\in\boA$).
If $\boA$ is an associative $R$-algebra and $\lambda\in\Hom_{\Ralg}(\boA,R)$ is a homomorphism of $R$-algebras, then $(\boA, \lambda)$  will be called an \emph{augmented algebra}. 
Note that $\lambda$ defines a left $\boA$-module $R_\lambda$ which is, as $R$-module, equal to $R$ and satisfies $a.r=\lambda(a)r$ for $a\in\boA$ and $r\in R_\lambda$.

\subsection{Traces and trace functions}
\label{ss:trace}
Let $\boA$ be an associative $R$-algebra. 
A  homomorphism of $R$-modules $\tilde\tau\colon \boA\to R$ satisfying $\tilde\tau(ab)=\tilde\tau(ba)$ for all $a,b\in \boA$ will be called a \emph{trace} on $\boA$. 
Let $[\boA,\boA]=\spn_R(\{\,ab-ba\mid a,b\in \boA\,\})$, and let $\uboA$ denote the $R$-module $\boA/[\boA,\boA]$%
\footnote{In the standard literature (cf. \cite{bass--eccdg}, \cite{bass--tec}, \cite{brown--cg}) this $R$-module is denoted by $T(\boA)$.}.
Then every trace $\tilde\tau$ induces a \emph{trace function} $\tau\colon \uboA\to R$. 

Let $\caB$ be a free basis of $\boA$ as $R$-module with $1\in \caB$, and let $\tilde\mu\colon \boA\to R$ be the $R$-linear function defined by 
\begin{equation}\label{eq:mu}
\tilde\mu_\caB(b)=\begin{cases}
1&\text{ if }b=1,\\
0&\text{ otherwise.}
\end{cases}\end{equation}
If $\tilde\mu_\caB$ is a trace, then it induces the \emph{canonical trace function} $\mu_\caB\in\Hom_R(\uboA,R)$.

\subsection{Euler algebras}
\label{ss:eaec} 
Let $\boA$ be an associative $R$-algebra.
A left $\boA$-module $M$ is called \emph{of type \FP}, if it has a finite, projective resolution $(P_\bullet,\partial_\bullet^P,\eps_M)$, $\eps_M \colon P_0\to M$, by finitely generated projective left $\boA$-modules, i.e., there exists a positive integer $m$ such that $P_k=0$ for $k>m$ or $k<0$, and $P_k$ is finitely generated for all $k$.
An augmented, associative $R$-algebra $\boA=(\boA, \lambda)$ is called to be \emph{of type \FP} if the left $\boA$-module $R_\lambda$ is of type \FP. 
By definition (cf.~\S\ref{s:intro}), an $R$-Euler algebra $(\boA, \caB, \lambda)$ is an augmented $R$-algebra $(\boA,\lambda)$ of type \FP, for which the linear function $\tilde\mu\colon \boA\to R$ associated with $\caB$ is a trace (cf.~\eqref{eq:mu}).

\subsection{The Hattori-Stallings trace map}
\label{ss:hatsta}
For a finitely generated, projective, left $\boA$-module $P$ let $P^\ast=\Hom_{\boA}(P,\boA)$. 
Then $P^\ast$ carries canonically the structure of a right $\boA$-module, and it is also finitely generated and projective. 
One has a canonical isomorphism $\gamma_P\colon P^\ast\otimes_{\boA} P\longrightarrow \End_{\boA}(P)$ given by $\gamma_P(p^\ast\otimes p)(q)=p^\ast(q) p$, $p^\ast\in P^\ast$, $p,q\in P$ (cf. \cite[Chap.~I, Prop.~8.3]{brown--cg}).
The {\it evaluation map} $\eval_P\colon P^\ast\otimes_{\boA} P\to\uboA$ is given by $\eval_P(p^\ast\otimes p)=p^\ast(p)+[\boA,\boA]$. 
The map
\begin{equation}
  \label{eq:hatsta}
  \trace_P=\eval_P\circ\gamma_P^{-1}\colon\End_{\boA}(P)\longrightarrow\uboA
\end{equation}
is called the \emph{Hattori--Stallings trace map on $P$} and $r_P=\trace_P(\iid_P)\in\uboA$ is called the \emph{Hattori--Stallings rank} of $P$ (cf.~\cite[Chap.~IX.2]{brown--cg}, \cite{stallings--cg}). 
In particular, $\trace_P$ is $R$-linear, and for $f,g\in\End_{\boA}(P)$ one has 
\begin{equation}
  \label{eq:hatstatr}
  \trace_P(f\circ g)=\trace_P(g\circ f).
\end{equation}
From the elementary properties of the evaluation map one concludes that if $P_1$ and $P_2$ are two finitely generated projective left $\boA$-modules, one has
\begin{equation}
  \label{eq:addr}
  r_{P_1\oplus P_2}=r_{P_1}+r_{P_2}.
\end{equation}
Let $e \in \boA$, $e=e^2$, be an idempotent in the $R$-algebra $\boA$. 
Then  $\boA e$ is a finitely generated, projective, left $\boA$-module, and 
\begin{equation}\label{eq:r-idemp}
  r_{\boA e}=e+[\boA,\boA].
\end{equation}

\subsection{Finite, projective chain complexes}
\label{ss:ccs}
A chain complex $P=(P_\bullet, \der^P_\bullet)$ of left $\boA$-modules will be called \emph{finite} if $\{\, k\in\Z\mid P_k\not=0\,\}$ is finite and $P_k$ is finitely generated for all $k\in \Z$. 
Moreover, $P$ will be called \emph{projective}, if $P_k$ is projective for all $k$. 

For $P=(P_\bullet, \der^P_\bullet)$ and $Q=(Q_\bullet, \der^Q_\bullet)$ finite, projective chain complexes of left $\boA$-modules we denote by $(\dHom_{\boA}(P,Q)_\bullet,d_\bullet)$ the chain complex of right $\boA$-modules
\begin{equation}
  \label{eq:Chomcom}
  \dHom_{\boA}(P,Q)_k=
  \prod_{j=i+k}\Hom_{\boA}(P_i,Q_j),
\end{equation}
with differential $d_k\colon \dHom_{\boA}(P,Q)_k\to \dHom_{\boA}(P,Q)_{k-1}$ given by
\begin{equation}
  \label{eq:Chomcom2}
  (d_k(f_k))_{i,j-1}=
  \partial^Q_{j}\circ f_{i,j}-(-1)^k f_{i-1,j-1}\circ \partial_i^P,
\end{equation}
for $f_k=\sum_{j=i+k}f_{i,j}$.
In particular, $f_0=\sum_{i\in\Z} f_{i,i}\in\dHom_{\boA}(P,Q)_0$ is a chain map of degree $0$ if, and only if, $f_0\in\kernel(d_0)$, and $f_0$ is homotopy equivalent to the $0$-map if, and only if, $f_0\in\image(d_1)$ (cf. \cite[Chap.~I]{brown--cg}).
Put $\dExt^{\boA}_0(P,Q)=H_{0}(\dHom_{\boA}(P,Q))$.

Let $B=(B_\bullet,\partial_\bullet^B)$ be a finite, projective chain complex of right $\boA$-modules. 
Then $(B\utimes_{\boA} P,\partial_\bullet^{\utimes})$ denotes the complex
\begin{equation}
  \label{eq:Cdeften}
  \begin{gathered}
    (B\utimes_{\boA} P)_k=\coprod_{i+j=k} B_i\otimes_{\boA} P_j,\\
    \partial_{i+j}^{\utimes}(b_i\otimes p_j)=\partial_i^B(b_i)\otimes p_j +(-1)^i b_i\otimes \partial^P_j(p_j).
  \end{gathered}
\end{equation}
Let $\boA\dbl 0\dbr$ denote the chain complex of left $\boA$-modules concentrated in degree $0$ with $\boA\dbl 0\dbr_0=\boA$, and let $\uboA\dbl 0\dbr$ denote the chain complex of $R$-modules concentrated in degree $0$ with $\uboA\dbl 0\dbr_0=\uboA$.
Then $P^\circledast=(P^\circledast_\bullet,\partial_\bullet^{P^\circledast})= (\dHom_\boA(P,\boA\dbl 0\dbr)_\bullet,d_\bullet)$,
\begin{equation}
  \label{eq:Cdefhomcom}
  \begin{gathered}
    P^\circledast_{k}=\Hom_\boA(P_{-k},\boA),\\
    \partial_k^{P^\circledast}(p_k^\ast)(p_{1-k})=(-1)^{k+1} p^\ast_k(\partial^P_{1-k}(p_{1-k})).
  \end{gathered}
\end{equation}
is a finite, projective complex of right $\boA$-modules.
Note that the differential of the complex is chosen in such a way that the \emph{standard evaluation mapping}
\begin{equation}
  \label{eq:Csteva}
  \begin{gathered}
    \eval_P\colon
    P^\circledast \utimes_{\boA} P\longrightarrow \uboA\dbl 0\dbr,\\
    \eval_{s,t}(p^\ast_s\otimes p_t)=\delta_{s+t,0}\,\,p^\ast_s(p_t) ,
  \end{gathered}
\end{equation}
is a mapping of chain complexes.
However, the natural isomorphism
\begin{equation}
  \label{eq:Chomcomten}
  \begin{gathered}
    \gamma\colon
    \dHom_\boA(\argu_1,\boA\dbl 0\dbr)\utimes_\boA\argu_2\longrightarrow
    \dHom_\boA(\argu_1,
    \argu_2)\\
    \gamma_{s,t}(p_s^\ast\otimes_\boA q_t)(x_{-s})=(-1)^{st}\,
    p_s^\ast(x_{-s}) q_t
  \end{gathered}
\end{equation}
comes equipped with a non-trivial sign (cf. \cite[Chap.~I, Prop.~8.3(b) and Chap.~VI, \S6, Ex.~1]{brown--cg}).
In this context the \emph{Hattori--Stallings trace map} is given by
\begin{equation}
  \label{eq:tracecom}
  \trace_P=H_0(\eval_P\circ\gamma^{-1}_{P,P})\colon\dExt_0^{\boA}(P,P)\longrightarrow H_0(\uboA\dbl0\dbr) \simeq \uboA.
\end{equation}
It has the following properties:

\begin{prop}
  \label{prop:tracecc}
  Let $P=(P_\bullet,\der_\bullet^P)$ be a finite, projective complex of left $\boA$-modules, and let $[f],[g]\in\dExt_0^{\boA}(P,P)$, $f=\sum_{k\in\Z}f_k$, be homotopy classes of chain maps of degree $0$. 
  Then
  \begin{itemize}
  \item[(a)] $\trace_P([f])=\sum_{k\in\Z}(-1)^k \trace_{P_k}(f_k)$;
  \item[(b)] $\trace_P([f]\circ [g])=\trace_P([g]\circ [f])$.
  \item[(c)] Let $Q=(Q_\bullet,\der_\bullet^Q)$ be another finite, projective complex of left $\boA$-modules which is homotopy equivalent to $P$, i.e., there exist chain maps $\phi\colon P\to Q$, $\psi\colon Q\to P$, which composites are homotopy equivalent to the respective identity maps. 
    Let $[h]\in\dExt_0^{\boA}(Q,Q)$ such that $[\phi]\circ [f]=[h]\circ[\phi]$.  
    Then $\trace_P([f])=\trace_Q([h])$.
  \end{itemize}
\end{prop}

\begin{proof}
  Part (a) is a direct consequence of \eqref{eq:Chomcomten}, and (b) follows from (a) and \eqref{eq:hatstatr}.

  \noindent
  The left hand side quadrangle in the diagram
  \begin{equation}
    \label{dia:mass}
    \xymatrix{
      \dHom_{\boA}(P,P) \ar[d]_{\phi\circ\argu\circ\psi}& \ar[l]_-{\gamma}
      P^\circledast\utimes_{\boA} P\ar[r]^-{\eval_P}\ar[d]^{\psi^\circledast\otimes\phi}&
      \uboA\dbl 0\dbr\ar@{=}[d]\\
      \dHom_{\boA}(Q,Q)& \ar[l]_-{\gamma} 
      Q^\circledast\utimes_{\boA} Q\ar[r]^-{\eval_Q}&
      \uboA\dbl 0\dbr\\
    }
  \end{equation}
  commutes, and the right hand side quadrangle commutes up to homotopy equivalence. This yields claim~(c).
\end{proof}

Let $P=(P_\bullet,\partial_\bullet^P)$ be a finite, projective complex of left $\boA$-modules. 
Then one defines the \emph{Hattori--Stallings rank} of $P$ by
\begin{equation}
  \label{eq:hatstacom}
  r_P=\trace_P([\iid_P])=\textstyle{\sum_{k\in\Z}  (-1)^k r_{P_k}}\in\uboA.
\end{equation}
Proposition~\ref{prop:tracecc} implies that if $Q=(Q_\bullet,\partial_\bullet^Q)$ is another finite, projective, complex of left $\boA$-modules which is homotopy equivalent to $P$ then $r_P=r_Q$.

Let $\caK(\boA)$ denote the additive category the objects of which are finite, projective chain complexes of left $\boA$-modules.
Morphisms $\Hom_{\caK(\boA)}(P,Q)= \dExt_0^\boA(P,Q)$ are given by the homotopy classes of chain maps of degree $0$.
In particular, $\caK(\boA)$ is a \emph{triangulated category} and distinguished triangles are triangles isomorphic to the cylinder/cone triangles 
(cf. \cite{gelfand-manin--ha}, \cite[Chap.~10]{weibel--iha}). 
Thus, if
\begin{equation}
  \label{eq:distri}
  \xymatrix{
    A\ar[r]& B\ar[r]& C\ar[r]& A[1]
  }
\end{equation}
is a distinguished triangle in $\caK(\boA)$, one has $r_B=r_A+r_C$.

Let $M$ be a left $\boA$-module of type \FP, and let $(P_\bullet,\der_\bullet,\eps_M)$ be a finite, projective resolution.
In particular, $P=(P_\bullet,\der_\bullet)$ is a finite, projective chain complex of left $\boA$-modules.
One defines the Hattori--Stallings rank of $M$ by $r_M=r_P\in\uboA$.
The comparison theorem in homological algebra implies that this element is well defined.
The following property will be useful for our purpose.

\begin{prop}
  \label{prop:FPcc}
  Let $C=(C_\bullet,\der^C_\bullet)$ be a chain complex of left $\boA$-modules concentrated in non-negative degrees with the following properties:
  \begin{itemize}
  \item[(a)] $C$ has homology concentrated in degree zero, i.e., $H_k(C)=0$ for $k\in\Z$, $k>0$;
  \item[(b)] $C$ is finitely supported, i.e., $C_k=0$ for almost all $k\in\Z$;
  \item[(c)] $C_k$ is of type \FP\ for all $k\in\Z$.
  \end{itemize}
  Then $H_0(C)$ is of type \FP, and one has
  \begin{equation}
    \label{eq:idCs}
    r_{H_0(C)}=\textstyle{\sum_{k\geq 0} (-1)^k r_{C_k}\in\uboA.}
  \end{equation}
\end{prop}

\begin{proof}
  Let $\ell(C)=\min\{\,n\geq 0\mid C_{n+j}=0\ \text{for all $j\geq 0$}\,\}$ denote the length of $C$. 
  We proceed by induction on $\ell(C)$. 
  For $\ell(C)=1$, there is nothing to prove.
  Suppose the claim holds for chain complexes $D$, $\ell(D)\leq \ell-1$, satisfying the hypotheses (a)--(c), and let $C$ be a complex satisfying (a)--(c) with $\ell(C)=\ell$.
  Let $C^\wedge$ be the chain complex coinciding with $C$ in all degrees $k\in\Z\setminus\{0\}$ and satisfying $C^\wedge_0=0$. 
  Then $C^\wedge[-1]$ satisfies (a)--(c) and $\ell(C^\wedge[-1])\leq \ell-1$.
  Then, by induction, $M=H_1(C^\wedge)=H_0(C^\wedge [-1])$ is of type \FP, and $r_M=\sum_{k\geq 1} (-1)^{k+1} r_{C_k}$. 
  By construction, one has a short exact sequence of left $\boA$-modules $0\to M\overset{\alpha}{\to} C_0\to H_0(C)\to 0$.
  Let $(P_\bullet,\der_\bullet^P,\eps_M)$ be a finite, projective resolution of $M$, and let $(Q_\bullet,\der_\bullet^Q,\eps_{C_0})$ be a finite, projective resolution of $C_0$. 
  By the comparison theorem in homological algebra, there exists a chain map $\alpha_\bullet\colon P_\bullet \to Q_\bullet$ inducing $\alpha$. 
  Let $\Cone(\alpha_\bullet)$ denote the mapping cone of $\alpha_\bullet$.
  Then $(\Cone(\alpha_\bullet),\tilde{\der}_\bullet,\eps_\ast)$ is a finite, projective resolution of $H_0(C)$, i.e., $H_0(C)$ is of type \FP.
  Moreover, by the remark following \eqref{eq:distri} one has
  \begin{equation}
    \label{eq:indsum}
    r_{H_0(C)}=r_{\Cone(\alpha_\bullet)}=r_Q-r_P=r_{C_0}-r_M.
  \end{equation}
  This yields the claim.
\end{proof}

\subsection{The Euler characteristic of an Euler algebra}
\label{ss:eueuc}
Let $\boA=(\boA,\caB,\lambda)$ be an Euler $R$-algebra with canonical trace function $\mu_\caB\in\Hom_R(\uboA,R)$. 
The \emph{Euler characteristic} of $\boA$ is defined by
\begin{equation}
  \label{eq:eucharA}
  \chi_{(\boA,\caB,\lambda)}=\mu_\caB(r_{R_\lambda})\in R.
\end{equation}

\subsection{Induction}
\label{ss:ind}
Let $\boB\subseteq \boA$ be an $R$-subalgebra of $\boA$.
The canonical injection $j\colon \boB\to\boA$ induces a canonical map
\begin{equation}
  \label{eq:traceBA}
  \trace_{\boB/\boA}\colon\uboB\to\uboA.
\end{equation}
Induction $\Mind_{\boB}^{\boA} =\boA\otimes_{\boB}\argu$ is a covariant additive right-exact functor mapping finitely generated projective left $\boB$-modules to finitely generated projective left $\boA$-modules.
Moreover, if $\boA$ is a flat right $\boB$-module, then $\Mind_{\boB}^{\boA}$ is exact. 
Let $P$ be a finitely generated left $\boB$-module, and let $Q=\Mind_{\boB}^{\boA}(P)$. 
Then one has a canonical map $\iota\colon P\to Q$, $\iota(p)=1\otimes p$, which is a homomorphism of left $\boB$-modules.
As induction is left adjoint to restriction, every map $f\in\End_{\boB}(P)$ induces a map $\iota_\circ(f)=(\iota\circ f)_\ast\in\End_{\boA}(Q)$.

Let $P^\ast=\Hom_\boB(P,\boB)$ and $Q^\ast=\Hom_{\boA}(Q,\boA)$.
Then for $f\in P^\ast$ one has an induced map $\iota_\ast(f)=(j\circ f)_\ast\in Q^\ast$ making the diagram
\begin{equation}
  \xymatrix{
\End_{\boB}(P)\ar[d]_{\iota_\circ}&P^\ast\otimes_{\boB} P
\ar[l]_-{\gamma_P}\ar[d]^{\iota_\ast\otimes\iota}\ar[r]^-{\eval_P} &\uboB\ar[d]^{\trace_{\boB/\boA}}\\
\End_{\boA}(Q)&Q^\ast\otimes_{\boA}Q
\ar[l]_-{\gamma_Q}\ar[r]^-{\eval_Q}&\uboA
  }
\end{equation}
commute. This shows the following.

\begin{prop}
  \label{prop:ind}
  Let $\boB\subseteq\boA$ be an $R$-subalgebra of $\boA$ such that $\boA$ is a flat right $\boB$-module, and let $M$ be a left $\boB$-module of type \FP.
  Then $\Mind_{\boB}^{\boA}(M)$ is of type \FP, and one has
  \begin{equation}
    \label{eq:Mindr}
    r_{\Mind^{\boA}_{\boB}(M)}=\trace_{\boB/\boA}(r_M).
  \end{equation}
\end{prop}

Let $(\boA, \caB,\lambda)$ be an augmented, associative, $R$-algebra with a distinguished free $R$-basis $\caB$ such that the map $\tilde\mu_\caB$ defined in \eqref{eq:mu} is a trace. 
Let $\boB\subseteq\boA$ be an $R$-subalgebra of $\boA$ such that
\begin{itemize}
\item[(i)] $\boA$ is a flat right $\boB$-module;
\item[(ii)] the $R$-module $\boB$ is generated by $\caC=\caB\cap\boB$.
\end{itemize}
Then $(\boB,\caC, \lambda\vert_\boB)$ is an augmented, associative, $R$-algebra with a distinguished $R$-basis $\caC$ containing $1$ such that the map $\tilde\mu_\caC$ is a trace.
Let $\mu_\caB\colon\uboA\to R$ and $\mu_\caC\colon\uboB\to R$ denote the associated trace functions. 
Then one has a commutative diagram
\begin{equation}
  \label{eq:diatritr}
  \xymatrix{
    \uboB\ar[0,2]^{\trace_{\boB/\boA}}\ar[dr]_{\mu_\caC}&&\uboA\ar[dl]^{\mu_{\caB}}\\
    &R&
  }
\end{equation}
implying the following direct consequence of Proposition~\ref{prop:ind}.

\begin{cor}
  \label{cor:ind}
  Let $(\boA, \caB,\lambda)$ be an augmented, associative $R$-algebra with a distinguished $R$-basis $\caB$ such that the map $\tilde\mu_\caB$ is a trace, and let $\boB\subseteq\boA$ be an $R$-subalgebra satisfying {\rm (i)--(ii)}. 
  Let $M$ be a left $\boB$-module of type FP. 
  Then $\mu_\caC(r_M)=\mu_\caB(r_{\Mind_\boB^\boA(M)})$.
\end{cor}

\section{The Euler characteristic of a Hecke algebra}
\label{s:trhec}

\subsection{The canonical trace of a Hecke algebra}
\label{ss:cantr}
As a corollary of the proof of \cite[Prop. 8.1.1]{geck-pfeiffer--cfcgiha}, one obtains the following property.

\begin{prop}
  \label{prop:mar}
  Let $\caH$ be the $R$-Hecke algebra associated with the finitely generated Coxeter group $(W,S)$. 
  Let $\caB=\{\, T_w\mid w\in W\,\}$ be the standard basis, and let $\tilde\mu=\tilde\mu_\caB$ be the associated $R$-linear map given by \eqref{eq:mu}.
  Then, $\tilde\mu$ is a trace.
\end{prop}

\begin{rem}
  \label{rem:trace}
  The trace $\tilde\mu\colon\caH\to R$ can be seen as the \emph{canonical trace} on $\caH$. 
  It is straightforward to verify that for Hecke algebras of type $A_n$, $B_n$ or $D_n$ this trace coincides with the Jones--Ocneanu trace evaluated in $0$ (cf.~\cite{geck:trace}).
\end{rem}

\subsection{Properties of the Deodhar complex}
\label{ss:coxcomp}
Let $(W,S)$ be a spherical Coxeter group, and let $q\in R$ be such that $p_{(W,S)}(q)\in R^\times$. 
Then $R_q\simeq \caH e_S$ (cf. Prop.~\ref{prop:eI}); in particular, $R_q$ is a projective left $\caH$-module.
This shows that for any Coxeter group $(W,S)$ and $I\subseteq S$ such that $W_I$ is finite, the left $\caH$-module $\Mind_{\caH_I}^\caH(R_q)$ is finitely generated and projective.
As a consequence one has the following (cf. \cite[\S 6.8]{humphreys--rgcg}):

\begin{prop}
  \label{prop:latt}
  Let $(W,S)$ be a finitely generated Coxeter group, which is either affine or co-compact hyperbolic (cf. \cite[Ch. 6]{humphreys--rgcg}), and let $q\in R$ be such that $p_{(W_I,I)}(q)\in R^\times$ for any proper parabolic subgroup $(W_I,I)$.
  Then the Deodhar complex $(C_\bullet,\der_\bullet,\eps)$ together with the map $\eps\colon C_0\to R_q$ (cf. Rem.~\ref{rem:complex}) is a finite, projective resolution of $R_q$.
\end{prop}

\subsection{The Euler characteristic of a Hecke algebra}
\label{ss:euhec}
Combining the results of \S\ref{s:Euler} with the properties of the Deodhar complex we obtain the following result.

\begin{thm}\label{thm:AA}
Let $(W,S)$ be a finitely generated Coxeter group, let $R$ be a commutative ring with unit, and let $q\in R$ be such that the Poincar\'e polynomial $p_{(W_I,I)}(q)$ is invertible in $R$ for any spherical parabolic subgroup $(W_I,I)$.
Then $(\caH,\caB,\eps_q)$, where $\caH=\caH_q(W,S)$ is the $R$-Hecke algebra associated with $(W,S)$ and parameter $q$, $\caB=\{T_w\mid w\in W\}$ and $\eps_q$ is the index representation, is an $R$-Euler algebra. 
Moreover, one has
\begin{align}
\chi_{(\caH,\caB,\eps_q)}&=(p_{(W,S)}(q))^{-1}\in R\label{eq:eu1}\\
\intertext{and, if $(W,S)$ is not spherical, then}
\chi_{(\caH,\caB,\eps_q)}&= \sum_{I\subsetneq S} (-1)^{|S\setminus I|-1}\cdot\chi_{(\caH_I,\caB\cap\caH_I,\eps_q)},\label{eq:eu2}
\end{align}
In particular, if $R =\Z\dbl q\dbr$, then $\chi_{(\caH,\caB,\eps_q)}=p_{(W,S)}(q)^{-1}$.
\end{thm}

\begin{proof}
  First we show that $\caH$ is an $R$-Euler algebra. 
  We proceed by induction on $d=|S|$.
  As $p_{(W_I,I)}(q)\in R^\times$ for any spherical parabolic subgroup $(W_I,I)$, the left $\caH$-module $\Mind_{\caH_I}^\caH(R_q)$ is finitely generated and projective. 
  For $d\leq 2$, $(W,S)$ is spherical or affine, and in this case there is nothing to prove (cf. Prop.~\ref{prop:latt}). 

  Assume that the claim holds for all Coxeter groups $(W_J,J)$ with $|J|<d$.
  If $(W,S)$ is spherical, then, by hypothesis and Proposition~\ref{prop:eI}, the left $\caH$-module $R_q$ is projective and the claim follows. 
  Therefore, we may also assume that $(W,S)$ is not spherical. 
  By induction, for $K\subsetneq S$ the left $\caH_K$-module $R_q$ is of type \FP. 
  Hence $\Mind_{\caH_K}^\caH(R_q)$ is a left $\caH$-module of type \FP.
  Thus $C_k$ is a left $\caH$-module of type \FP\ for $0\leq k\leq d-1$, where $C=(C_\bullet,\der_\bullet)$ is the Deodhar complex of $\caH$. 
  From Theorem~\ref{thm:complex}(c) and Proposition~\ref{prop:FPcc} one concludes that $R_q$ is a left $\caH$-module of type \FP.
  Hence Proposition~\ref{prop:mar} implies that $(\caH,\caB,\eps_q)$ is an $R$-Euler algebra.

  In case that $(W,S)$ is spherical one has $R_q\simeq \caH e_S$, where $e_S$ is given as in Proposition~\ref{prop:eI}. 
  Hence $r_{R_q}=e_S+[\caH,\caH]$ (cf. \eqref{eq:r-idemp}), and thus 
  \begin{equation}\label{eq:chi-p}
    \chi_{(\caH,\caB,\eps_q)}=\mu(r_{R_q})=(p_{(W,S)}(q))^{-1}\in R,
  \end{equation}
  which yields \eqref{eq:eu1}.

  Suppose that $(W,S)$ is not spherical. 
  From \eqref{eq:idCs} and \eqref{eq:Mindr} one concludes that the Hattori--Stallings rank of $R_q$ satisfies the identity
  \begin{equation}
    \label{eq:hs1}
    r_{R_q}=\sum_{0\leq k< |S|} (-1)^k r_{C_k}
    =\sum_{I\subsetneq S} (-1)^{|S\setminus I|-1} r_{\Mind_{I}^{S}(R_q)}.
  \end{equation}
  Applying the trace function $\mu_\caB$ (cf.~Cor.~\ref{cor:ind}) one obtains \eqref{eq:eu2}.
  The identity \eqref{eq:eu1} is then a direct consequence of \eqref{eq:altsfor} and \eqref{eq:eu2}.
  If $R=\Z\dbl q\dbr$, then the Poincar\'e polynomial of any spherical parabolic subgroup $(W_I,I)$ is invertible since its constant term is equal to $1$. 
  Hence the initial hypothesis of the theorem is satisfied.
\end{proof}

%%%%%%%%%%%%%%%%%%%%%%%%%%%%%%%%%%% 
%% References
\bibliographystyle{amsplain}
\bibliography{eulhec-refs}
\end{document}